\documentclass[12pt,reqno]{amsart}
\usepackage{enumerate}
\usepackage{amsrefs}
\usepackage{amsfonts, amsmath, amssymb, amscd, amsthm, bm, cancel}
\usepackage{url}
\usepackage{graphicx}
\usepackage[
linktocpage=true,colorlinks,citecolor=magenta,linkcolor=blue,urlcolor=magenta]{hyperref}
\usepackage{multicol}
\usepackage{comment}
\usepackage[margin=1in]{geometry}
\newtheorem{thm}{Theorem}[section]

  \newtheorem{lem}{Lemma}[section]

 \theoremstyle{definition}

 \theoremstyle{remark}

\newtheorem{rem}{Remark}[section]

 \numberwithin{equation}{section}

\allowdisplaybreaks

\newcommand{\f}{\left(}
\renewcommand{\r}{\right)}

\renewcommand{\d}{\delta}

\newcommand{\s}{\sigma}

\begin{document}

\title{Interior Hessian estimates for Hessian quotient equations}

\author{Weisong Dong}
\address{School of Mathematics, Tianjin University,
	Tianjin, 300354, China}
\email{dr.dong@tju.edu.cn}

\author{Ruijia Zhang}
\address{Department of Mathematics, Sun Yat-sen University, Guangzhou, 510275, China}
\email{zhangrj76@mail.sysu.edu.cn}



\begin{abstract}
In this paper, we establish interior $C^2$ estimates for admissible semiconvex solutions to the general Hessian quotient equation
$
\frac{\sigma_k}{\sigma_l}(D^2u)=f(x,u),
$
for the cases $l=k-1$ and $l=k-2$, where $f$ is a positive $C^2$ function. Such estimates are known to fail in general for $k-l\geq 3$, even for convex solutions, as shown by counterexamples due to Lu \cite{LuGeneral}. The main ingredient is a quantitative concavity inequality for the Hessian quotient operator under the semiconvex condition. Our result provides a unified argument to such general Hessian quotient equations for $2\leq k\leq n-1$ in arbitrary dimensions.
\end{abstract}

\maketitle

\section{Introduction}\label{sec:1}
We consider a function $u$ satisfying the Hessian quotient equation
\begin{align}\label{eq hq}
F_l(D^2u):=\frac{\sigma_k}{\sigma_l}(D^2u)=f(x,u),
\qquad l=k-1\ \text{or}\ l=k-2,
\end{align}
in the admissible cone, namely
$\lambda(D^2u)\in\Gamma_k$. Our purpose is to prove interior $C^2$ estimates for admissible semiconvex solutions to \eqref{eq hq}, with a general right-hand side $f(x,u)$. Interior Hessian estimates are a central issue in the regularity theory of fully nonlinear elliptic equations. These estimates do not depend on the boundary data. Once such estimates are derived, standard elliptic theory yields higher interior regularity. We refer to \cite{CC} for the general background.

The study of interior $C^2$ estimates goes back to Heinz's work \cite{Heinz} on the Monge--Amp\`ere equation in $\mathbb{R}^2$. In dimensions $n\geq3$, Pogorelov \cite{Pogorelov} constructed singular convex solutions showing that an unconditional interior Hessian estimate for the Monge--Amp`ere equation fails. Urbas \cite{Urbas} extended these counterexamples to the $k$-Hessian equation $\sigma_k(D^2u)=f$, when $k\geq3$. Thus, by contrast, the quadratic Hessian equation occupies a distinguished position. Warren and Yuan \cite{WarrenYuan09} proved the interior Hessian estimate for the constant $\sigma_2$ equation in dimension three, which was later extended by Qiu \cite{Qiu} to general positive right-hand sides. In arbitrary dimensions, McGonagle--Song--Yuan \cite{McGonagleSongYuan} obtained interior Hessian estimates for almost convex solutions of the constant $\sigma_2$ equation by a compactness argument; Guan and Qiu \cite{GuanQiu} derived the estimate for admissible solutions under the additional condition $\sigma_3>-A$ using a pointwise maximum principle. Shankar and Yuan \cite{ShankarYuan20} proved the estimate for semiconvex solutions of the constant $\sigma_2$ equation via an integral method. Later, they \cite{ShankarYuan25} established the estimate without convexity assumption in $\mathbb{R}^4$ and, in higher dimensions, obtained estimates under a dynamic semiconvexity condition. Fan \cite{Fan} extended these results to variable right-hand sides. Very recently, Chen, Jian, Tu, and Zhou \cite{ChenJianTuZhou} proved interior $C^2$ regularity for convex viscosity solutions of $\sigma_2(D^2u)=f(x)$ with $f\in C^{0,1}$.

A closely related source of Hessian estimates is the special Lagrangian equation. In low dimensions, for certain phases, the special Lagrangian structure enters Hessian quotient equations. For the Hessian quotient equation $\frac{\sigma_3}{\sigma_1}(D^2u)=f,$
several interior Hessian estimates are known in dimensions three and four. Chen, Warren, and Yuan \cite{ChenWarrenYuan} proved the estimate for convex solutions when the right-hand side is constant. Wang and Yuan \cite{WangYuan14} later obtained interior estimates on the corresponding elliptic branch in dimensions three and four without assuming convexity. Zhou \cite{Zhou} extended these results to positive Lipschitz right-hand sides depending on $x$. In the curvature setting, Qiu and Zhou \cite{QiuZhou} obtained interior curvature estimates at the critical phase; in dimension four, their result covers admissible solutions of $\sigma_3(\kappa)=\sigma_1(\kappa)$. In dimension three, Lu \cite{LuDimensionThree} gave another proof using a Jacobi inequality and the Legendre transform.

For the general Hessian quotient equation,
$
\frac{\sigma_k}{\sigma_l}(D^2u)=f,$
Lu \cite{LuGeneral} proved interior $C^2$ estimates in the cases $k=n$ and $l=n-1,n-2$; Lu and Tsai \cite{LuTsaiPogorelov} subsequently established Pogorelov-type estimates and sharp conditional regularity results. In \cite{LuGeneral}, Lu constructed singular solutions when $k-l\geq3$. These results identify $l=k-1$ and $l=k-2$ as the natural remaining cases. In this paper, we verify the structural concavity condition introduced by Lu and Tsai \cite{Lu} under the weaker semiconvexity assumption.

\begin{lem}\label{general concavity}
Let $l=k-1$ or $l=k-2$, and let $F=\sigma_k/\sigma_l$. Suppose that
$\lambda\in\Gamma_k^n$, $\lambda_1\geq\cdots\geq\lambda_n\geq-K$, and that
$F$ has a fixed positive lower bound $f_0$. Then there exist large $K_0,K_1\geq1$
and small $\delta,\gamma>0$ such that, whenever
\[
\frac{\lambda_1}{F^{1/(k-l)}}\geq K_0,
\]
every $\zeta\in\mathbb R^n$ satisfies
\begin{equation}\label{c hq}
\begin{aligned}
&-\sum_{i,j=1}^nF^{ii,jj}\zeta_i\zeta_j
+\frac{K_1}{F}\left(\sum_iF^{ii}\zeta_i\right)^2
+\frac{2}{(1+\delta)\lambda_1}\sum_{i>1}F^{ii}\zeta_i^2 \geq
(1+\gamma)\frac{F^{11}\zeta_1^2}{\lambda_1}.
\end{aligned}
\end{equation}
The constants depend only on $n,k,l,K$ and $f_0$.
\end{lem}

The concavity inequality in Lemma \ref{general concavity} can be traced back to an algebraic calculation of Huisken and Sinestrari \cite{HS}, which gives a useful recursive structure for elementary symmetric functions. Zhang \cite{Zhang} later developed this idea into a quantitative concavity inequality for the $k$-Hessian operator, providing effective control of third-order terms when the largest eigenvalue is large. Guan and Sroka \cite{GuanSroka} also found a special concavity property for positive Hessian quotient operators with $k=n$. More recently, Dong, Xu and Zhang \cite{DXZ} introduced a lifting argument and extended the $k$-Hessian concavity method to the dynamically semiconvex setting, where negative eigenvalues may still have a significant size. This idea was later used by Mei and Yan \cite{MeiYan} for the semiconvex equations $\sigma_3/\sigma_l=1$, $l=1,2$. Related interior estimates were also obtained by Jiao and Sui \cite{JiaoSui} for $\sigma_2/\sigma_1=\psi(x,u)$, while Fung \cite{FungHQ} developed a pointwise doubling argument for Hessian quotient equations under suitable concavity assumptions. Motivated by these developments, we adapt the lifting idea to $\sigma_k/\sigma_l$, $l=k-1,k-2$, and use a tangential--radial decomposition to remove the restriction to tangent vectors. This gives the main algebraic input for the Jacobi inequality.

Our proof then follows the integral approach originating from Warren and Yuan \cite{WarrenYuan09}, which was later refined by Qiu \cite{QiuAJM} so that the Sobolev inequality is no longer needed. Shankar and Yuan \cite{ShankarYuan20} further developed this method for semiconvex solutions of the $\sigma_2$ equation, while Lu and Tsai \cite{Lu} applied a similar argument to convex Hessian quotient equations. The main difficulty in the semiconvex setting is that a large positive eigenvalue may coexist with negative eigenvalues of non-negligible size. Besides Lemma \ref{general concavity}, we verify the uniform ellipticity of the transformed operator after the Lewy--Legendre transform, which allows us to extend the convex interior estimates to admissible semiconvex solutions in arbitrary dimensions. Our main result is stated below.

\begin{thm}\label{thm main}
    Let $n\geq 3$ and $1\leq l<k\leq n-1$, with $l=k-1$ or $l=k-2$. Let $u\in C^4(B_{9})$ be an admissible solution of \eqref{eq hq}, where $f\in C^2(B_{9}\times \mathbb{R})$ is positive. Assume that $u$ is semiconvex, i.e.
    \[D^2 u\geq -KI\]
    for some $K\geq 0$. Then we have 
    \[|D^2 u(0)|\leq C,\]
where $C$ depends on $n$, $k$, $K$, $\|u\|_{L^{\infty}( B_8)}$, $\min_{\bar B_7\times[-M,M]}f$ and $\|f\|_{C^2(\bar B_7\times[-M,M])}.$ Here $M$ is a large constant satisfying $\|u\|_{L^{\infty}( B_8)}\leq M.$
\end{thm}
\begin{rem}
The case $n=3$, $k=2$, and $l=1$ was already resolved by Jiao and Sui \cite{JiaoSui}; for $k=3$ and $f=1$, the corresponding semiconvex interior estimates were obtained by Mei and Yan \cite{MeiYan}. Most recently, for $l=k-1$  the convex case was settled by Tsai \cite{TsaiConcavity}, where he proved the required concavity inequality, \eqref{c hq}  by a change of basis for symmetric polynomials; for $l=k-2$, Li and Wu \cite{LW} derived the results by establishing \eqref{c hq} with $l=k-1,\ k-2$ in the convex cone via a different approach.
\end{rem}

The paper is organized as follows. In Section \ref{sec:2}, we collect and prove some algebraic properties of Hessian quotient functions, and verify Lemma \ref{general concavity}. In Section \ref{sec:3}, we establish the crucial Jacobi inequality, perform the Lewy--Legendre transform, and derive the weighted mean-value estimate. In Section \ref{sec:4}, we complete the proof of Theorem \ref{thm main} by integration by parts.

\section{Preliminaries}\label{sec:2}
\subsection{Symmetric function}
Here we state some algebraic properties of the elementary symmetric functions $\s_m(\lambda)$, $m=1,\cdots,n$, where $ \lambda=( \lambda_1,\cdots, \lambda_n)$. Recall that $\s_m(\lambda)$ is defined by
	\begin{equation*}
	\s_m(\lambda)=\sum_{1\leqslant i_1<\cdots<i_m \leqslant n}\lambda_{i_1}\cdots\lambda_{i_m}, \quad m=1,2, \cdots,n.
	\end{equation*}
	The Garding cone $\Gamma_m$ is an open symmetric convex cone in $\mathbb{R}^n$ with vertex at the origin, given by
	\begin{align}\label{gk}
	\Gamma_m=\lbrace\f \lambda_1,\cdots, \lambda_n\r\in\mathbb{R}^n |\s_j(\lambda)>0,  \forall j= 1,\cdots,m\rbrace.
	\end{align}
	Clearly $\s_m(\lambda)=0$ for $\lambda \in\partial \overline{\Gamma}_m$ and
	\begin{align*}
	\Gamma_n\subset\cdots\subset {\Gamma}_m\subset\cdots\subset\Gamma_1.
	\end{align*}
	In particular, $\Gamma^+ = \Gamma_n$ is called the positive cone,
	\begin{align*}
	\Gamma^+=\lbrace\f \lambda_1,\cdots, \lambda_n\r\in\mathbb{R}^n |\lambda_1>0,\cdots,\lambda_n>0\rbrace.
	\end{align*}
	We always assume $\lambda_1\geq\lambda_2\geq\cdots\geq \lambda_n$. We collect some properties of $\s_k(\lambda)$ where $\lambda\in\Gamma_k$:
 \begin{enumerate}[(1)]
     \item\label{P1} $\sum_{p=1}^n\frac{\partial \s_k}{\partial \lambda_p}\lambda_p^2=\s_1\s_k-(k+1)\s_{k+1}$. 
     \item\label{P2} $\sum_{p=1}^n \frac{\partial \s_k}{\partial \lambda_p}=(n-k+1)\sigma_{k-1}.$
     \item\label{lem'} $ \frac{\partial \s_k}{\partial \lambda_n}\geq \cdots\geq \frac{\partial \s_k}{\partial \lambda_1}>0.$
     \item\label{k1} For any $s<k$, $\s_s>\lambda_1\lambda_2\cdots\lambda_s$.
 \end{enumerate}

 We point out that we will denote $C$, $c$ or $C_i$ and $c_i$ for $i=1,2,\dots,n$ as some positive constants throughout the subsequent sections, and these constants may change from line to line.
\begin{lem}\label{lemn}\cite{Zhang}
   Assume that $\lbrace\lambda_i\rbrace\in \Gamma_k$ and $\lambda_1\geq\lambda_2\geq \cdots\geq \lambda_n$. Then $\lambda_k>0$ and $|\lambda_n|<(n-k+1)\lambda_k.$  
 \end{lem}
\begin{lem}\label{lll}
    Assume that $\lambda\in \Gamma_k$ and $\lambda_n>-K$. If $c_0\leq \frac{\s_k}{\s_l}\leq C_0$ with $l=k-1$ or $l=k-2$, then we have 
    \[ \lambda_{l+1}\geq c,\quad \lambda_k\leq C\]
    where the constants depend on $n$, $k$, $c_0$, $C_0$ and $K$.
 \end{lem}
\begin{proof}
For $l=k-2$, since $\sigma_k=f\sigma_{k-2}$ and $f$ has fixed positive upper and lower bounds,
\[\lambda_1\dots\lambda_{k-1}(\lambda_k+c_{n,k}\lambda_n)\leq \s_k=f\s_{k-2}\leq C\lambda_1\dots\lambda_{k-2},\]
we have $\lambda_k\leq C(n,k,K).$
Using Lemma \ref{lemn}, we have
 \[C\lambda_1\dots\lambda_{k-1}\lambda_k\geq \s_k=f\s_{k-2}\geq c\lambda_1\dots\lambda_{k-2},\]
 Thus 
 \begin{align}\label{k-1 k}
     \lambda_{k-1}\geq c(n,k,c_0,C_0, K),\quad \lambda_k\leq C(n,k,c_0,C_0, K).
 \end{align}
By a similar argument, using $\sigma_k=f\sigma_{k-1}$, we obtain $c\leq \lambda_k\leq C$ for $l=k-1$.
\end{proof}
After the Lewy--Legendre transform in Section \ref{sec:3}, we need the operator to be uniformly elliptic, which is a purely algebraic property. We therefore state it here. We also introduce the following notation,
\begin{equation}\label{G coefficient}
 G^{ii}=F^{ii}(\lambda_i+K+1)^2.
\end{equation}
\begin{lem}\label{ell}
For $l=k-1$ or $l=k-2$, $G^{ii},\forall 1\leq i\leq n$ is uniformly elliptic after a proper normalization, i.e.,
\[
 c\leq G^{ii}\leq C,\qquad l=k-1,
\]
and
\[
 c\frac{\sigma_{k-1}}{\sigma_k}\leq G^{ii}\leq
 C\frac{\sigma_{k-1}}{\sigma_k},\qquad l=k-2.
\]
where the constants depend only on $n,k,K$ and the positive upper and lower bounds for
$f$.
\end{lem}

\begin{proof}

We note that by the Newton inequality,
    \begin{align}\label{fii 1}
\frac{F^{ii}}{F}=\frac{\sigma_k^{ii}}{\sigma_k}-\frac{\sigma_{k-1}^{ii}}{\sigma_{k-1}}=\frac{\sigma_{k-1;i}^2-\sigma_{k;i}\sigma_{k-2;i}}{\sigma_k\sigma_{k-1}}\geq c(n,k)\frac{\sigma_{k-1;i}^2}{\sigma_k\sigma_{k-1}},\quad l=k-1,
    \end{align}
    \begin{align}\label{fii 2}
   \frac{F^{ii}}{F}=\frac{\sigma_k^{ii}}{\sigma_k}-\frac{\sigma_{k-2}^{ii}}{\sigma_{k-2}}=\frac{\sigma_{k-1;i}\sigma_{k-2;i}-\sigma_{k;i}\sigma_{k-3;i}}{\sigma_k\sigma_{k-2}}\geq c(n,k)\frac{\sigma_{k-1;i}\sigma_{k-2;i}}{\sigma_k\sigma_{k-2}},\quad l=k-2,
    \end{align}

and 
\begin{align}\label{fii 0}
     (k-1)\frac{\s_{k-1}}{\s_k}>\sum_{i=1}^n\frac{F^{ii}}{F}=\sum_{i=1}^n\frac{\sigma_k^{ii}}{\sigma_k}-\sum_{i=1}^n\frac{\sigma_{k-1}^{ii}}{\sigma_{k-1}}=(n-k+1)\frac{\sigma_{k-1}}{\sigma_k}-(n-k+2)\frac{\sigma_{k-2}}{\sigma_{k-1}}\geq c(n,k)\frac{\sigma_{k-1}}{\sigma_k}.
    \end{align}
For $l=k-1$, by using \eqref{fii 1} we have 
\begin{align}\label{u1}
     \lambda_i^2\frac{F^{ii}}{F}\leq C(n,k,K),\quad 1\leq i\leq k-1, 
\end{align}
and 
\begin{align}\label{u2}
     \lambda_i^2\frac{F^{ii}}{F}\leq C(n,k,K),\quad i\geq k, 
\end{align}
For $l=k-2$, by using \eqref{fii 2} we have 
\begin{align}\label{u3}
     \lambda_i^2\frac{F^{ii}}{F}\leq C(n,k,K)
     \frac{\s_{k-1}}{\s_k},\quad 1\leq i\leq k-2, 
\end{align} 
\begin{align}\label{u4}
     \frac{F^{ii}}{F}\leq C(n,k,K)\frac{\s_{k-1}}{\s_k},\quad i\geq k, 
\end{align}
and 
\begin{align}\label{u5}
     \lambda_{k-1}^2\frac{F^{k-1,k-1}}{F}<C(n,k,K)\lambda_{k-1}\frac{(\s_{k}-\s{k;k-1})\s_{k-2}}{\s_k\s_{k-2}}\leq C(n,k)\frac{\s_{k-1}}{\s_k}.
\end{align}
\textit{Case (a)} If $\sigma_{k;i}>0$ for some $1\leq i\leq k-1$, then 
\[ \sigma_{k-1;i}>\frac{\lambda_1\cdots\lambda_{k-1}\lambda_k}{\lambda_i},\quad \sigma_{k-2;i}>\frac{\lambda_1\cdots\lambda_{k-1}}{\lambda_i}.\]
Thus, plugging into \eqref{fii 1} and \eqref{fii 2}, we have
\begin{align}\label{l1}
  \lambda_i^2 \frac{F^{ii}}{F}\geq c(n,k)\frac{\sigma_k^2}{\sigma_k\sigma_{k-1}}=c(n,k),\quad l=k-1
\end{align}
and 
\begin{align}\label{l2}
 \lambda_i^2 \frac{F^{ii}}{F}\geq c(n,k)\frac{\sigma_k\sigma_{k-1}}{\sigma_k\sigma_{k-2}}=c(n,k)\frac{\s_{k-1}}{\sigma_{k-2}},\quad l=k-2.
\end{align}
\\
 If $\sigma_{k;i}>0$ for some $ i\geq k$, then
\[\sigma_{k-1;i}>\lambda_1\cdots\lambda_{k-1},\quad  \sigma_{k-2;i}>\lambda_1\cdots\lambda_{k-2}.\]
Thus, plugging into \eqref{fii 1} and \eqref{fii 2}, we have
\begin{align}\label{l3}
 \frac{F^{ii}}{F}\geq c(n,k)\frac{\sigma_{k-1}^2}{\sigma_k\sigma_{k-1}}=c(n,k),\quad l=k-1
\end{align}
and 
\begin{align}\label{l4}
\frac{F^{ii}}{F}\geq c(n,k)\frac{\sigma_{k-1}\sigma_{k-2}}{\sigma_k\sigma_{k-2}}=c(n,k)\frac{\s_{k-1}}{\sigma_{k-2}},\quad l=k-2.
\end{align}
\textit{Case (b)} If $\sigma_{k;i}\leq 0$ for some $ 1\leq i\leq n$, then
\[\lambda_i\sigma_{k-1;i}\geq\s_k,\quad \sigma_{k-2;i}>\lambda_1\cdots\lambda_{k-2}.\]
Thus, plugging into \eqref{fii 1}, we have
\begin{align}\label{l5}
 \lambda_i^2\frac{F^{ii}}{F}\geq c(n,k)\frac{\sigma_{k}^2}{\sigma_k\sigma_{k-1}}=c(n,k),\quad l=k-1.
\end{align}
When $l=k-2$, if $\lambda_i\s_{k-2;i}\geq \frac{\s_{k-1}}{2}$, then 
\begin{align}\label{l8}
     \lambda_i^2\frac{F^{ii}}{F}\geq c(n,k)\frac{\sigma_k\sigma_{k-1}}{\sigma_k\sigma_{k-2}}=c(n,k)\frac{\s_{k-1}}{\s_k};
\end{align}
if $\lambda_i\s_{k-2;i}<\frac{\s_{k-1}}{2}$, then
\[\s_{k-1;i}=\s_{k-1}-\lambda_i\s_{k-2;i}>\frac{\s_{k-1}}{2}.\]

For $1\leq i\leq k-2,$ 
\[\s_{k-2;i}\geq \frac{\lambda_1\cdots\lambda_{k-2}}{\lambda_i}.\]
Plugging into \eqref{fii 2}, we have
\begin{align}\label{l6}
     \lambda_i\frac{F^{ii}}{F}\geq c(n,k)\frac{\sigma_{k-1}\sigma_{k-2}}{\sigma_k\sigma_{k-2}}=c(n,k)\frac{\s_{k-1}}{\s_k};
\end{align}
For $i\geq k-1,$ 
\[\s_{k-2;i}\geq \lambda_1\cdots\lambda_{k-2}. \]
Plugging into \eqref{fii 2} we have
\begin{align}\label{l7}
    \frac{F^{ii}}{F}\geq c(n,k)\frac{\sigma_{k-1}\sigma_{k-2}}{\sigma_k\sigma_{k-2}}=c(n,k)\frac{\s_{k-1}}{\s_k}.
\end{align}
Since $\lambda_i+K+1$ is comparable to $\lambda_i$ when $\lambda_i$ is large and is bounded above and below otherwise, combining \eqref{u1}, \eqref{u2}, \eqref{u3}, \eqref{u4}, \eqref{u5}, \eqref{l1}, \eqref{l2}, \eqref{l3}, \eqref{l4}, \eqref{l5}, \eqref{l8}, \eqref{l6} and \eqref{l7}, together with $F=f$, we obtain 
\begin{align}
    C\geq G^{ii}\geq c,\quad l=k-1
\end{align}
and 
\begin{align}
    C\frac{\s_{k-1}}{\s_k}\geq G^{ii}\geq c\frac{\s_{k-1}}{\s_k},\quad l=k-2.
\end{align}
\end{proof}
We first recall the following inequality in \cite{DXZ}. See also \cites{HZ, Zhang}.
\begin{lem}[\cite{DXZ}, Lemma 3.1]\label{DXZ}
Suppose that $\lambda\in\Gamma_k^n$ and $\lambda_n\geq-\delta \lambda_1$. Suppose that $\delta\in(0,1)$ is sufficiently small and $\lambda_1/\s_k^{\frac{1}{k}}$ is sufficiently large
depending only on $n$ and $k$. Then for every $\xi\in\mathbb R^n$, there exist some $\gamma>0$ and $K_1>0$ depending on $n$ and $k$, such that
\begin{equation}
\begin{aligned}
 &-\frac{\displaystyle\sum_{p\ne q}
       \sigma_k^{pp,qq}\xi_p\xi_q}{\sigma_k}
 +K_1\frac{\displaystyle\left(\sum_i\sigma_k^{ii}\xi_i\right)^2}
          {\sigma_k^2}
 +\frac{2}{1+\delta}\sum_{i>1}\frac{\sigma_k^{ii}\xi_i^2}
 {\lambda_1\sigma_k}  \\
 &\hspace{35mm}\geq
 (1+\gamma)\frac{\sigma_k^{11}\xi_1^2}
 {\lambda_1\sigma_k}.
\end{aligned}
\end{equation}
\end{lem}

Using Lemma \ref{DXZ}, we establish the following concavity inequality for the Hessian quotient operator $ F(\lambda)=\frac{\sigma_k(\lambda)}{\sigma_l(\lambda)}$ with $l=k-1$ and $l=k-2$. Denote $F^{ii}=\partial F/\partial\lambda_i$ and
$F^{ii,jj}=\partial^2F/(\partial\lambda_i\partial\lambda_j)$ for vectors tangent to the level set of $F$.
\begin{lem}\label{sc iq}
Let $1\leq l< k\leq n$ and $l=k-1$, or $l=k-2$. Suppose that $\lambda\in\Gamma_k^n$, $\lambda_n\geq-K$ and that $F$ has a fixed positive positive lower bound $f_0$. There exists a large $K_0\geq 1$ and small constants $\delta>0$ and $\gamma>0$, depending only on $n$, $k$, $K$ and $f_0$, such that if
$\lambda_1/F^{\frac{1}{k-l}}\geq K_0$, then for every $\xi\in\mathbb R^n$ satisfying
\begin{equation}\label{cri eq}
 \sum_iF^{ii}\xi_i=0,
\end{equation}
the estimate
\begin{equation}\label{sc hq}
 -\sum_{i,j=1}^nF^{ii,jj}\xi_i\xi_j
 +\frac{2}{(1+\delta)\lambda_1}
  \sum_{i>1}F^{ii}\xi_i^2
 \geq
 (1+\gamma)\frac{F^{11}\xi_1^2}{\lambda_1}
\end{equation}
holds.
\end{lem}

\begin{proof}
Since $F$ is homogeneous of degree $k-l$, \eqref{sc hq} is homogeneous of degree $k-l-2$. After replacing $\lambda$ by
$\lambda/F(\lambda)^{\frac{1}{k-l}}$ it is enough for us to prove the assertion under the following normalization
\begin{equation}\label{F1}
 F=1,
 \qquad \sigma_k=\sigma_{l}.
\end{equation}
For $0<s<1$, set
\[
 \tilde\lambda=(\lambda_1,\ldots,\lambda_n,(k-l-1),-1+s),
 \qquad
 \tilde\xi=(\xi_1,\ldots,\xi_n,0,0).
\]
For $1\leq j\leq k$,
\begin{equation}\label{sj}
 \sigma_j(\tilde\lambda)
 =\sigma_j-\sigma_{j-k+l}+s(\s_{j-1}+(k-l-1)\s_{j-2})
\end{equation}
where $\sigma_0=1$ and $\sigma_j=0$ for $j<0$.
The Newton--Maclaurin inequalities imply that,
\begin{equation}
 \frac{\sigma_j}{\sigma_{j-k+l}}\geq \frac{(l+1)(n-j+k-l)\sigma_k}{(j+1-k+l)(n-l)\sigma_l}>1.
\end{equation}
Thus, for $2\leq j<k$,
\[
\sigma_j(\tilde\lambda)>s\sigma_{j-1}>0;
\]
for $ j=1$,
\[\sigma_1(\tilde\lambda)=\sigma_1+s+k-l-2>0;\]
for $j=k$, equations \eqref{F1} and \eqref{sj} give
\begin{equation}\label{k hq2}
 \sigma_k\bigl(\tilde\lambda(s)\bigr)
 =s\left(\sigma_{k-1}+(k-l-1)\sigma_{k-2}\right)>0.
\end{equation}
Therefore, for every $0<s<1$,
\begin{equation}\label{sk2}
     \tilde\lambda(s)\in\Gamma_k^{n+2}.
\end{equation}

Choose $C_0=\delta^{-1}(K/f_0^{\frac{1}{k-l}}+1)$. Here $\delta$ is determined by Lemma \ref{DXZ}. Under the normalized assumption, for $0<s<1$, 
\[
 \min_{1\leq p\leq n+2}\lbrace \tilde\lambda_p(s)\rbrace \ge-\frac{K}{f_0^{\frac{1}{k-l}}}-1.
\]
Denote by $K_0'=\max\lbrace C_0, K_0\rbrace$ where $K_0$ is determined by Lemma \ref{DXZ}.
If
\begin{equation}
 \frac{\lambda_1}{\sigma_k(\tilde\lambda(s))^{1/k}}
 =\frac
  {\lambda_1}{
   \left[s(\sigma_{k-1}+(k-l-1)\sigma_{k-2})\right]^{1/k}}\geq C(n,k)\lambda_1^{\frac{1}{k}}\ge K_0',
\end{equation}
in other words,
\[\lambda_1\geq C(n,k) K_0'^{}k,\] then $\tilde \lambda$ satisfies the assumptions of Lemma \ref{DXZ}. We replace $C(n,k) K_0'^k$ by $K_0$ for convenience in this Lemma. Applying Lemma \ref{DXZ} to $\tilde\lambda$ yields
\begin{equation}\label{sk tl}
\begin{aligned}
 &-\displaystyle\sum_{p\ne q}
        \sigma_k(\tilde \lambda)^{pp,qq}\xi_p\xi_q
 +K_1\frac{\displaystyle\left(\sum_i \sigma_k(\tilde \lambda)^{ii}\xi_i\right)^2}
          { \sigma_k(\tilde \lambda)}
 +2\sum_{i>1}\frac{ \sigma_k(\tilde \lambda)^{ii}\xi_i^2}
 {(1+\delta)\lambda_1 }  \\
 &\hspace{35mm}\geq
 (1+\gamma)\frac{ \sigma_k(\tilde \lambda)^{11}\xi_1^2}
 {\lambda_1 }.
\end{aligned}
\end{equation}

We rewrite \eqref{k hq2} as
\begin{equation}\label{eq s}
 \sigma_k(\tilde\lambda)
 =\sigma_l(F-1)+s\bigl(\sigma_{k-1}+(k-l-1)\sigma_{k-2}\bigr)=s\bigl(\sigma_{k-1}+(k-l-1)\sigma_{k-2}\bigr).
\end{equation}
Direct differentiation gives,
\begin{equation}
 \sigma_k^{ii}(\tilde\lambda)
 =\sigma_l F^{ii}+s\bigl(\sigma_{k-1}^{ii}+(k-l-1)\sigma_{k-2}^{ii}\bigr)
\end{equation}
for $1\leq i\leq n$. Under \eqref{cri eq}, since $\xi_{n+1}=\xi_{n+2}
=0$, we have
\begin{equation}\label{eq 2.1}
 \sum_{p=1}^{n+2}\sigma_k^{pp}(\tilde\lambda)\tilde\xi_p
 =s\sum_i(\sigma_{k-1}^{ii}+(k-l-1)\sigma_{k-2}^{ii})\xi_i.
\end{equation}
For $1\leq i, j\leq n$,
\[
 \sigma_k^{ii,jj}(\tilde\lambda)
 =\sigma_{l}F^{ii,jj}+\sigma_{l}^{ii}F^{jj}+\sigma_{l}^{jj}F^{ii}+s(\sigma_{k-1}^{ii,jj}+(k-l-1)\sigma_{k-2}^{ii,jj}).
\]
Under \eqref{cri eq}, since $\xi_{n+1}=\xi_{n+2}
=0$, we have
\begin{equation}\label{eq2 1}
\begin{aligned}
 \sum_{p\ne q}\sigma_k^{pp,qq}(\tilde\lambda)
       \tilde\xi_p\tilde\xi_q
 &=\sigma_l\sum_{i,j}F^{ii,jj}\xi_i\xi_j+s\sum_{i\ne j}(\sigma_{k-1}^{ii,jj}+(k-l-1)\sigma_{k-2}^{ii,jj})\xi_i\xi_j.
\end{aligned}
\end{equation}
Similarly,
\begin{equation}\label{eq2 2}
 \sum_{p>1}\sigma_k^{pp}(\tilde\lambda)\tilde\xi_p^2
 =\sigma_l\sum_{i>1}F^{ii}\xi_i^2
 +s\sum_{i>1}(\sigma_{k-1}^{ii}+(k-l-1)\sigma_{k-2}^{ii})\xi_i^2,
\end{equation}
and
\begin{equation}\label{eq2 3}
 \sigma_k^{11}(\tilde\lambda)\tilde\xi_1^2
 =\left(\sigma_l F^{11}+s(\sigma_{k-1}^{11}+(k-l-1)\sigma_{k-2}^{11})\right)\xi_1^2.
\end{equation}
Plugging \eqref{eq s}, \eqref{eq 2.1}, \eqref{eq2 1}, \eqref{eq2 2} and \eqref{eq2 3} into \eqref{sk tl}, we obtain
\begin{equation}\label{eq2 hq1}
\begin{aligned}
 &-\sigma_l\sum_{i,j}F^{ii,jj}\xi_i\xi_j
 -s\sum_{i\ne j}(\sigma_{k-1}^{ii,jj}+(k-l-1)\sigma_{k-2}^{ii,jj})\xi_i\xi_j\\
 &+\frac{K_1s}{\sigma_{k-1}+(k-l-1)\sigma_{k-2}}
 \left(\sum_i(\sigma_{k-1}^{ii}+(k-l-1)\sigma_{k-2}^{ii})\xi_i\right)^2 \\
 &
 +\frac{2\sigma_l}{(1+\delta)\lambda_1}
 \sum_{i>1}F^{ii}\xi_i^2
 +\frac{2s}{(1+\delta)\lambda_1}
 \sum_{i>1}(\sigma_{k-1}^{ii}+(k-l-1)\sigma_{k-2}^{ii})\xi_i^2 \\
\geq&
 \frac{1+\gamma}{\lambda_1}
 \left(\sigma_l F^{11}+s(\sigma_{k-1}^{11}+(k-l-1)\sigma_{k-2}^{11})\right)\xi_1^2.
\end{aligned}
\end{equation}

Dividing \eqref{eq2 hq1} by $\sigma_l$ gives
\begin{equation}
\begin{aligned}
 &-\sum_{i,j}F^{ii,jj}\xi_i\xi_j
 +\frac{2}{(1+\delta)\lambda_1}
 \sum_{i>1}F^{ii}\xi_i^2 -(1+\gamma)\frac{F^{11}\xi_1^2}{\lambda_1}\\
 &+s\Bigg[
 -\sum_{i\ne j}\frac{\sigma_{k-1}^{ii,jj}+(k-l-1)\sigma_{k-2}^{ii,jj}}
 {\sigma_{l}}\xi_i\xi_j
 +K_1\frac{\left(\sum_i(\sigma_{k-1}^{ii}+(k-l-1)\sigma_{k-2}^{ii})\xi_i\right)^2}
 {\sigma_{l}(\sigma_{k-1}+(k-l-1)\sigma_{k-2})}\\
 &\hspace{17mm}
 +\frac{2}{(1+\delta)\lambda_1}
 \sum_{i>1}\frac{(\sigma_{k-1}^{ii}+(k-l-1)\sigma_{k-2}^{ii})\xi_i^2}{\sigma_{l}}
 -(1+\gamma)\frac{(\sigma_{k-1}^{11}+(k-l-1)\sigma_{k-2}^{11})\xi_1^2}
 {\lambda_1\sigma_{l}}\Bigg]\\
 &\geq 0
\end{aligned}
\end{equation}
for all $0<s<1$. Thus we obtain \eqref{sc hq} by letting $s\rightarrow0$.

\end{proof}

\begin{rem}
    By using Lemma \ref{DXZ}, we can establish the same concavity inequalities as in Lemma \ref{sc iq} under the dynamic semiconvexity condition.
\end{rem}

Now we apply a suitable tangential--radial decomposition to remove the tangency assumtion in Lemma \ref{sc iq}, thereby establishing the Lu–Tsai concavity assumption \cite{Lu} under the semiconvexity condition.\\

\textit{Proof of Lemma \ref{general concavity}}

Decompose
\begin{equation}\label{de z}
 \zeta=\xi+a\lambda,
 \qquad
 a=\frac{\sum_iF^{ii}\zeta_i}{(k-l)F}.
\end{equation}
Thus,
\begin{align}\label{ce}
\sum_iF^{ii}\xi_i=0.     
\end{align}

Since $F$ is homogeneous of degree $k-l$,
\[
 \sum_jF^{ii,jj}\lambda_j=(k-l-1)F^{ii},
 \qquad
 \sum_{i,j}F^{ii,jj}\lambda_i\lambda_j=(k-l)(k-l-1)F.
\]
Consequently, by \eqref{ce},
\begin{equation}\label{radial hessian}
 -\sum_{i,j}F^{ii,jj}\zeta_i\zeta_j=-\sum_{i,j}F^{ii,jj}(\xi_i+a\lambda_i)(\xi_j+a\lambda_j)
 =-\sum_{i,j}F^{ii,jj}\xi_i\xi_j-(k-l)(k-l-1)a^2F.
\end{equation}
Using Young's inequality, we have
\begin{equation}\label{radial gradient}
 \frac{2}{(1+\delta)\lambda_1}\sum_{i>1}F^{ii}\zeta_i^2=\frac{2}{(1+\delta)\lambda_1}\sum_{i>1}F^{ii}(\xi_i+a\lambda_i)^2
 \geq
 \frac{2}{(1+2\delta)\lambda_1}\sum_{i>1}F^{ii}\xi_i^2
 -C_\delta a^2\sum_{i>1}\frac{F^{ii}\lambda_i^2}{\lambda_1}.
\end{equation}
Apply Lemma \ref{sc iq} with $2\delta$ in place of $\delta$. Lemma \ref{ell} implies that for every $1\leq i\leq n$
\begin{equation}\label{radial bounds1}
F^{ii}\lambda_i^2\leq CF,\quad l=k-1
\end{equation}
\begin{equation}\label{radial bounds2}
\frac{F^{ii}\lambda_i^2}{\lambda_1}\leq CF\frac{\s_{k-1}}{\lambda_1\s_k}\leq CF,\quad l=k-2.
\end{equation}
It follows from \eqref{radial hessian}, \eqref{radial bounds1} and \eqref{radial bounds2} that the constrained
left-hand side evaluated at $\xi$ is at most
\begin{equation}\label{lk1}
 -\sum_{i,j}F^{ii,jj}\zeta_i\zeta_j
 +\frac{2}{(1+\delta)\lambda_1}\sum_{i>1}F^{ii}\zeta_i^2
 +Ca^2F.
\end{equation}
 On the other hand, Young's inequality, \eqref{radial bounds1} and \eqref{radial bounds2} give
\begin{equation}\label{lk2}
 (1+\gamma_0)\frac{F^{11}(\zeta_1-a\lambda_1)^2}{\lambda_1}
 \geq
 \left(1+\frac{\gamma_0}{2}\right)\frac{F^{11}\zeta_1^2}{\lambda_1}
 -C_{\gamma_0}a^2F.
\end{equation}
Finally,
\[
 a^2F=\frac{1}{(k-l)^2F}\left(\sum_iF^{ii}\zeta_i\right)^2.
\]
Combining \eqref{lk1} and \eqref{lk2}, and choosing $K_1$ sufficiently large depending on $n$, $k$, $K$ and $f_0$, we prove Lemma \ref{general concavity} after relabeling $\gamma_0/2$ as $\gamma$.
\qed

\section{Jacobi inequality}\label{sec:3}

\begin{lem}\label{jacobi}
Let $u$ be a smooth admissible solution of \eqref{eq hq} satisfying
$\lambda_n\geq-K$. Then there exist a large $K_0$, a small $\varepsilon>0$, and
$C>0$, depending only on the data in Theorem \ref{thm main}, such that, at a
point $x_0$ where $\lambda_1>K_0$,
\begin{equation}\label{jacobi eq}
 F^{ij}b_{ij}\geq \varepsilon F^{ij}b_i b_j-C,
 \qquad b=\log\lambda_1,
\end{equation}
holds in the viscosity sense.
\end{lem}

\begin{proof}
In this proof, $f_i$ and $f_{ij}$ denote the total derivatives of
$f(x,u(x))$. We first work at a point $x_0$ where $D^2u$ is diagonal and
$\lambda_1> K_0 (\sup F)^{1/(k-l)}.$ 
Choose a canonical coordinate frame $e_1, \ldots, e_n$ at $x_0$ such that $\{u_{ij} (x_0)\}$ is diagonal. Let the eigenvalues $\lambda$ of $D^2 u$ be ordered as
$$\lambda_1(x_0)=\lambda_2(x_0)=\dots=\lambda_m(x_0)>\lambda_{m+1}(x_0)>\dots>\lambda_n(x_0).$$ We construct a unit vector field $v(x)$ near $x_0$ such that $v(x_0)=e_1(x_0)$ and $D_i v(x_0)=\sum_{p\neq 1}\frac{u_{1pi}}{\lambda_1+1-\lambda_p}e_p$. Take 
\[\phi(x)=\ln (D^2 u(v(x),v(x))+\langle v(x), e_1\rangle^2-1)\]
Then $\phi(x)\leq b(x)$ and locally $\phi(x_0)= b(x_0)$.
By using the maximum principle, we have
\begin{equation}\label{b first}
b_i=\frac{u_{11i}}{\lambda_1},
\end{equation}
and
\begin{equation}\label{b second}
b_{ii}
\geq
\frac{u_{11ii}}{\lambda_1}
+
\frac{2}{\lambda_1}
\sum_{p>1}\frac{u_{1pi}^2}{\lambda_1-\lambda_p+1}
-
\frac{u_{11i}^2}{\lambda_1^2}.
\end{equation}
Multiplying by $F^{ii}$ and summing over $i$, we have 
\[F^{ii}b_{ii}
\geq
F^{ii}\frac{u_{11ii}}{\lambda_1}
+
\frac{2}{\lambda_1}
\sum_{p>1}\frac{F^{pp}u_{pp1}^2}{\lambda_1-\lambda_p+1}+\frac{2}{\lambda_1}\sum_{1<i\leq m}F^{11}u_{11i}^2+\frac{2}{\lambda_1}\sum_{m+1\leq p\leq n}\frac{F^{11}u_{11p}^2}{\lambda_1-\lambda_p+1}
-
F^{ii}\frac{u_{11i}^2}{\lambda_1^2}.\]
Differentiating \eqref{eq hq} once and twice in the $x_1$ direction gives
\begin{equation}\label{first diff}
 \sum_pF^{pp}u_{pp1}=f_1,
\end{equation}
and
\begin{equation}\label{second diff}
 \sum_iF^{ii}u_{ii11}+F^{ij,rs}u_{ij1}u_{rs1}=f_{11}.
\end{equation}
Here
\[
 f_{11}=f_{x_1x_1}+2f_{x_1u}u_1+f_{uu}u_1^2+f_u u_{11}\geq -C(1+\lambda_1).
\]

Due to Andrews \cite{And}, we have
\[
 F^{ij,rs}u_{ij1}u_{rs1}
 =\sum_{p,q}F^{pp,qq}u_{pp1}u_{qq1}
 +2\sum_{\lambda_p<\lambda_q}\frac{F^{pp}-F^{qq}}{\lambda_p-\lambda_q}u_{pq1}^2\leq \sum_{p,q}F^{pp,qq}u_{pp1}u_{qq1}
 +2\sum_{p\geq m+1}\frac{F^{pp}-F^{11}}{\lambda_p-\lambda_1-1}u_{11p}^2.
\]
Contracting \eqref{b second} with $F^{ii}$, we obtain
\begin{equation}\label{jacobi calculation}
\begin{aligned}
 F^{ii}b_{ii}\geq{}&
 \frac{1}{\lambda_1}\left[
 -\sum_{p,q}F^{pp,qq}u_{pp1}u_{qq1}
 +2\sum_{p>1}\frac{F^{pp}u_{pp1}^2}{\lambda_1-\lambda_p+1}
 \right] +\frac{2}{\lambda_1}\sum_{1<i\leq m}F^{ii}u_{11i}^2\\
 &+\sum_{p\geq m+1}\frac{F^{pp}}{\lambda_1^2}
 \left(\frac{2\lambda_1}{\lambda_1-\lambda_p+1}-1\right)u_{11p}^2
 -\sum_{1\leq i\leq m}\frac{F^{11}u_{11i}^2}{\lambda_1^2}-C
\end{aligned}
\end{equation}
where the corresponding $F^{11}$-terms cancel.
Apply Lemma \ref{general concavity} with $\zeta_p=u_{pp1}$. Combining \eqref{first diff} and \eqref{c hq}, we get 
\begin{align*}
 &\frac{1}{\lambda_1}\left[
 -\sum_{p,q}F^{pp,qq}u_{pp1}u_{qq1}
 +\frac{2}{1+\delta}\sum_{p>1}\frac{F^{pp}u_{pp1}^2}{\lambda_1}
 \right] \\
 &\qquad\geq
 (1+\gamma)\frac{F^{11}u_{111}^2}{\lambda_1^2}
 -\frac{K_1f_1^2}{F\lambda_1}
 \geq(1+\gamma)\frac{F^{11}u_{111}^2}{\lambda_1^2}-C.
\end{align*}
We used here that $F=f(x,u)$ is bounded below and that
$f_1=f_{x_1}+f_u u_1$ is bounded by the data. Moreover,
\[
 \frac{2\lambda_1}{\lambda_1-\lambda_p+1}-1
 =\frac{\lambda_1+\lambda_p-1}{\lambda_1-\lambda_p+1}
 \geq\frac{1-2\delta}{1+2\delta}\geq\frac13
\]
for some $\delta\leq \frac14$. We assume that $\lambda_1\geq \frac23$, and that 
\[\frac{2}{\lambda_1}\sum_{1<i\le m}F^{ii}u_{11i}^2\geq  \frac43\frac{F^{ii}u_{11i}^2}{\lambda_1}\]
Therefore,
\[
 F^{ii}b_{ii}\geq
 \gamma F^{11}b_1^2+\frac13\sum_{p>1}F^{pp}b_p^2-C
 \geq\varepsilon\sum_pF^{pp}b_p^2-C.
\]
\end{proof}
Now, we take a Lewy--Legendre transform of equation \eqref{eq hq}. Let
$\tilde u=u(x)+\frac{K+1}{2}|x|^2$, with
$w(y(x))=D\tilde u\cdot x-\tilde u$, and set
\[
 y=Du(x)+(K+1)x.
\]
Then $D^2w(y)=(D^2u+(K+1)I)^{-1}$ and
\[
 G(D^2w(y))=-F\left((D^2w(y))^{-1}-(K+1)I\right)
 =-f(x(y),u(x(y))).
\]
At a point where $D^2u$ is diagonal, the coefficients $G^{ii}$ are given by \eqref{G coefficient}.

\begin{lem}\label{sub}
Denote $b^*(y(x))=b(x)$. At the point $y(x_0)$, the following inequality holds:
\[
 G^{ii}b^*_{ii}\geq-C,
 \qquad b=\log\lambda_1,
\]
in the viscosity sense.
\end{lem}
\begin{proof}
At a diagonal point, the chain rule gives
\[
 b_i=(u_{ij}+(K+1)\d_{ij})b_j^*=(\lambda_i+K+1)b_i^*,
\]
and
\[
 b_{ii}=\sum_j u_{iij}b_j^*+(\lambda_i+K+1)^2b_{ii}^*.
\]
Using \eqref{G coefficient} and
$\sum_iF^{ii}u_{iij}=f_j$, we obtain
\[
 F^{ii}b_{ii}=\sum_jf_jb_j^*+G^{ii}b_{ii}^*.
\]
Lemma \ref{jacobi} therefore implies
\begin{align*}
 G^{ii}b_{ii}^*
 &\geq\varepsilon F^{ii}b_i^2-\sum_jf_jb_j^*-C\\
 &=\varepsilon G^{ii}(b_i^*)^2-\sum_jf_jb_j^*-C\\
 &\geq-C\sum_i\frac{1}{G^{ii}}-C.
\end{align*}
The estimates in Lemma \ref{ell} imply that
$\sum_i(G^{ii})^{-1}\leq C$, and we complete the proof.
\end{proof}

\
By Lemma \ref{ell}, the coefficients $G^{ii}$ are uniformly elliptic after the proper normalization. Thus, we apply the local maximum principle (Theorem 4.8 in \cite{CC}) and obtain
the following mean-value inequality.
\begin{lem}\label{l mv}
Let
\[
 \tilde b=\log\max\{\lambda_1,K_0\}.
\]
Then
\[
 \tilde b(0)\leq C\int_{B_1}\tilde b(x)\sigma_{k-1}(D^2u(x))\,\mathrm{d}x+C,
\]
where $C$ has the same dependence as in Theorem \ref{thm main}.
\end{lem}
The proof is adapted from \cite{Lu}; only the semiconvex shift and the variable
right-hand side require changes.
\begin{proof}
The maximum of two viscosity subsolutions is still a subsolution. Lemma
\ref{sub} gives
\[
 G^{ii}\tilde b^*_{ii}\geq-C.
\]
Lemma \ref{ell} implies that $\sum_iG^{ii}\geq c$. Thus, for a fixed sufficiently large
$A$,
\[
 G^{ii}(\tilde b^*+A|y|^2)_{ii}\geq0.
\]
Without loss of generality, we may assume $y(0)=0$. The local
maximum principle and Lemma \ref{ell} yield
\[
 \tilde b^*(0)\leq C\int_{B_1^y(0)}(\tilde b^*(y)+A|y|^2)\,\mathrm{d}y.
\]
Since $D^2u+(K+1)I\geq I$,
\[
 |y(x)-y(z)|\geq|x-z|.
\]
Applied to $B_2^x(0)$, this implies
$B_1^y(0)\subset y(B_2^x(0))$ and
$x(B_1^y(0))\subset B_1^x(0)$. Changing variables, we obtain
\[
 \tilde b(0)\leq C\int_{B_1}\tilde b(x)
 \det(D^2u+(K+1)I)\,\mathrm{d}x+C.
\]
Lemma \ref{lll}, with
$\sigma_k=f\sigma_l$, gives
\begin{equation}\label{sn}
 \det(D^2u+(K+1)I)\leq C\sigma_{k-1}+C.
\end{equation}
Therefore,
\begin{equation}\label{mv}
 \tilde b(0)\leq C\int_{B_1}\tilde b(x)\sigma_{k-1}\,\mathrm{d}x+C.
\end{equation}
\end{proof}
\section{Interior estimates}\label{sec:4}
In this section, we combine Lemma \ref{l mv} with integration by parts, following
\cite{Lu}. Set
\begin{equation}\label{H definition}
 H^{ij}=\sigma_lF^{ij}=\sigma_k^{ij}-f\sigma_l^{ij}.
\end{equation}
Since the Newton tensors are divergence-free,
\begin{equation}\label{H divergence}
 \sum_jD_jH^{ij}=-\sum_jf_j\sigma_l^{ij},
\end{equation}
where $f_j=D_j(f(x,u(x)))$. 

For completeness, we retain the algebraic lemma used in the induction.
\begin{lem}\label{alg}
Under the assumptions of Theorem \ref{thm main}, at a point where $D^2u$ is
diagonal, for every $1\leq j\leq k-1$,
\[
 (\sigma_j^{ii})^2\leq CH^{ii}\sigma_{k-1}.
\]
\end{lem}
\begin{proof}
Recall from Lemma \ref{lll} that $c\leq\lambda_{k-1}$ and $\lambda_k\leq C$.
For $1\leq i\leq k-1$,
\[
 (\sigma_j^{ii})^2\leq C\frac{\sigma_{k-1}^2}{\lambda_i^2},
\]
whereas for $i\geq k$,
\[
 (\sigma_j^{ii})^2\leq C\sigma_{k-1}^2.
\]
The lower bounds \eqref{l1}, \eqref{l3}, \eqref{l5} for $l=k-1$ and
\eqref{l2}, \eqref{l4}, \eqref{l8}, \eqref{l6}, \eqref{l7} for $l=k-2$
give
\[
 H^{ii}\geq c\frac{\sigma_{k-1}}{\lambda_i^2}
 \quad(1\leq i\leq k-1),
 \qquad
 H^{ii}\geq c\sigma_{k-1}
 \quad(i\geq k),
\]
which implies the assertion.
\end{proof}

\begin{proof}[Proof of Theorem \ref{thm main}]
Let $b=\log\max\{\lambda_1,K_0\}$, as in Lemma \ref{l mv}. Let
$\psi\in C_c^\infty(B_2)$ satisfy $\psi=1$ on $B_1$. Since
$\sigma_{k-1}^{ij}$ is divergence-free,
\begin{align}
 \int_{B_1}b\sigma_{k-1}\,\mathrm{d}x
 &\leq\frac1{k-1}\int_{B_2}\psi b\sigma_{k-1}^{ij}u_{ij}\,\mathrm{d}x \notag\\
 &=-\frac1{k-1}\int_{B_2}(\psi_i b+\psi b_i)
 \sigma_{k-1}^{ij}u_j\,\mathrm{d}x \notag\\
 &\leq C\int_{B_2}b\sigma_{k-2}\,\mathrm{d}x
 +C\int_{B_2}|b_i\sigma_{k-1}^{ii}|\,\mathrm{d}x.
\label{ig1}
\end{align}
Using Lemma \ref{alg} and repeating the same integration-by-parts step for
$b\sigma_j$, $1\leq j\leq k-2$, Lemma \ref{l mv} gives
\begin{equation}\label{ig2}
 b(0)\leq C\int_{B_3}H^{ij}b_i b_j\,\mathrm{d}x
 +C\int_{B_3}\sigma_{k-1}\,\mathrm{d}x+C.
\end{equation}

Multiplying Lemma \ref{jacobi} by $\sigma_l$ and using
$\sigma_l\leq C\sigma_{k-1}$ gives
\begin{equation}\label{weighted jacobi}
 H^{ij}b_{ij}\geq\varepsilon H^{ij}b_i b_j-C\sigma_{k-1}
\end{equation}
in the viscosity sense. The maximum defining $b$ preserves this inequality. By Ishii \cite{Ishii}, the above inequality also holds in the distribution sense. Let
$\phi\in C_c^\infty(B_4)$ satisfy $\phi=1$ on $B_3$. Then
\begin{align*}
 \int_{B_4}\phi^2H^{ij}b_i b_j\,\mathrm{d}x
 &\leq\frac1\varepsilon\int_{B_4}\phi^2H^{ij}b_{ij}\,\mathrm{d}x
 +C\int_{B_4}\sigma_{k-1}\,\mathrm{d}x\\
 &=-\frac2\varepsilon\int_{B_4}\phi\phi_jH^{ij}b_i\,\mathrm{d}x
 -\frac1\varepsilon\int_{B_4}\phi^2(D_jH^{ij})b_i\,\mathrm{d}x
 +C\int_{B_4}\sigma_{k-1}\,\mathrm{d}x\\
 &\leq\frac14\int_{B_4}\phi^2H^{ij}b_i b_j\,\mathrm{d}x
 +C\int_{B_4}H^{ij}\phi_i\phi_j\,\mathrm{d}x\\
 &\quad+C\int_{B_4}\phi^2|f_j\sigma_l^{ij}b_i|\,\mathrm{d}x
 +C\int_{B_4}\sigma_{k-1}\,\mathrm{d}x.
\end{align*}
The total derivatives $f_j=f_{x_j}+f_u u_j$ are bounded by the data. Applying Lemma
\ref{alg} with $j=l$ and a sufficiently small $\eta$, and using
\[
 \sum_iH^{ii}=(n-k+1)\sigma_{k-1}-f(n-l+1)\sigma_{l-1}
 \leq C\sigma_{k-1},
\]
we absorb the energy term and obtain
\begin{equation}\label{hii2}
 \int_{B_3}H^{ij}b_i b_j\,\mathrm{d}x
 \leq C\int_{B_4}\sigma_{k-1}\,\mathrm{d}x.
\end{equation}
Combining \eqref{ig2} and \eqref{hii2},
\begin{equation}\label{ig3}
 b(0)\leq C\int_{B_4}\sigma_{k-1}\,\mathrm{d}x+C.
\end{equation}

Finally, if $\Phi\in C_c^\infty(B_5)$ equals one on $B_4$, then
\[
 \int_{B_4}\sigma_{k-1}\,\mathrm{d}x
 \leq-\frac1{k-1}\int_{B_5}\Phi_i\sigma_{k-1}^{ij}u_j\,\mathrm{d}x
 \leq C\int_{B_5}\sigma_{k-2}\,\mathrm{d}x.
\]
Repeating this argument gives
\[
 b(0)\leq C\int_{B_7}\sigma_1\,\mathrm{d}x+C\leq C,
\]
where the last inequality follows by one final integration by parts and the gradient bound for $u$. Together with $D^2u\geq-KI$, $\|u\|_{C^{0,1}(\bar B^7)}$ can be controlled by $\|u\|_{L^{\infty}(B^8)}$. Hence $\lambda_1(0)\leq C$, which proves the theorem.
\end{proof}

\end{document}